\documentclass[12pt,twoside,reqno,centertags,draft]{amsart}
\usepackage[top=3cm, bottom=3cm, left=3cm, right=3cm] {geometry}
\usepackage{color}
\usepackage{amsmath,amsthm,amsfonts,amssymb}
\numberwithin{equation}{section}
\allowdisplaybreaks
\newtheorem{theorem}{Theorem}[section]
\newtheorem{proposition}{Proposition}[section]

\newtheorem{example}{Example}[section]

\newtheorem{remark}{Remark}[section]

\newtheorem{definition} {{Definition}}[section]

\def\R{{\mathbb R}}
\newcommand{\be} {\begin{equation}}
\newcommand{\ee} {\end{equation}}
\newcommand{\bea} {\begin{eqnarray}}
\newcommand{\eea} {\end{eqnarray}}
\newcommand{\Bea} {\begin{eqnarray*}}
\newcommand{\Eea} {\end{eqnarray*}}

\newcommand{\esssup}{\mathop{\rm {ess\,sup}}\limits}    
\newcommand{\essinf}{\mathop{\rm {ess\,inf}}\limits}

\begin{document}

\title{\bf Pairs of fixed points for a class \\ of operators on Hilbert spaces}
\date{}
\maketitle

\begin{center}
{\sc Abdelhak Mokhtari$^{a,b}$, Kamel Saoudi$^{c,d}$, Du\v{s}an D. Repov\v{s}$^{e,f,g,}$\footnote{Corresponding author: Du\v{s}an D. Repov\v{s}}}
\end{center}

\vspace{0.6cm}

{\tiny
\begin{center}
$^a$Mathematics Department, Faculty of Mathematics and Informatics, University of M'sila, Algeria.\\
$^b$Laboratory of Fixed Point Theory and Applications, Department of Mathematics, E.N.S. Kouba, Algiers, Algeria; Email address: abdelhak.mokhtari@univ-msila.dz \\[0.3cm]
$^c$College of Sciences at Dammam, University of Imam Abdulrahman Bin Faisal, 31441 Dammam, Kingdom of Saudi Arabia \\
$^d$Basic and Applied Scientific Research Center, Imam Abdulrahman Bin Faisal University, P.O. Box 1982, 31441, Dammam, Kingdom of Saudi Arabia; Email address: kmsaoudi@iau.edu.sa\\[0.3cm]
$^e$Faculty of Education, University of Ljubljana, 1000 Ljubljana, Slovenia.  \\
$^f$Faculty of Mathematics and Physics, University of Ljubljana, 1000 Ljubljana, Slovenia.  \\
$^g$Institute of Mathematics, Physics and Mechanics, 1000 Ljubljana, Slovenia; Email address: dusan.repovs@guest.arnes.si
\end{center}

\vspace{0.3cm}

\textbf{Abstract.}
In this paper, existence of pairs of solutions is obtained for compact potential operators on Hilbert spaces.  An application to a second-order
boundary value problem is also given as an illustration of our results. 
 

 }

\section{Introduction}\label{s1}
There are many papers that study the existence of fixed points for different types of operators, among the most important of these, the  potential operator (or gradient operator)  which can be regarded as the G\^{a}teaux derivative of a suitable functional. Potential operators arise in many steady-state phenomena in physical problems  from quantum mechanics, e.g. the potential of the Hamiltonian operator in the Schr\"{o}dinger equation, see \cite{3,1,2}.

In fact, the variational methods to study linear and nonlinear equations are fully based on potential operators (see \cite{bouc, BMT2016, CS2011, CS2013, Ric2012}). In \cite{CS2011}, the authors considered nonlinear mappings $\phi\in C^1(H,\mathbb{R})$ defined on a Hilbert space $H,$ ordered by a cone $P$ and such that $\phi$ satisfies the Palais-Smale condition  and has the expression
 \begin{equation}\label{rr}
 \phi '= I-A.
 \end{equation}
 When $A$ satisfies some growth conditions,  $A$ was shown
  to have a fixed point. A combination of topological and variational methods were used and an application to a second-order dynamic equation was given.

In \cite{CS2013}, the authors discussed the existence of  fixed points for a class of nonlinear operators on Hilbert spaces with lattice structure, by a combination of variational and partial ordered methods and they gave an application to second-order ordinary differential equations.
In \cite{TWY2013}, the authors proved fixed point theorems for a widely more generalized hybrid non-self mappings on a Hilbert space. Using these results, they were able to prove the Browder-Petryshyn fixed point theorem \cite{BP1967} for strict pseudo-contractive non-self mappings and also  generalized the fixed point theorem from
 \cite{KTY2010} to the super hybrid non-self mappings.

Motivated by these previous works, we are concerned
in this paper
 with the existence of pairs of fixed points on a Hilbert space $H$ for an odd compact potential operator $A: H\longrightarrow H,$ satisfying the following sublinear growth condition
$$
(\mathcal{H})\quad \textrm{There exists}\;\theta\in[0,1)\,\mbox{ such that }\,\limsup_{\|u\|\rightarrow+\infty}\frac{\|Au\|}{\|u\|^{\theta}}<\infty.
$$
When $\theta=1$, this condition is known as the quasiboundedness of $A$. If $$\limsup\limits_{\|u\|\rightarrow+\infty}\frac{\|Au\|}{\|u\|^{\theta}}<1,$$ then existence of the fixed point of $A$ is guaranteed by an application of
Rothe's theorem (see, e.g., \cite{granas}).

 The proofs of our results are based on the critical point theory. In particular, we apply Clark's theorem on a functional associated with the operator $A$. We choose it in an appropriate way so that its critical points are the same as the fixed points of $A,$ as shown by
  relation (\ref{rr}). This method guarantees that the critical points on the boundary of the ball, that
  is different from the origin $0_H$, which assures that the associated fixed points
 are nontrivial, which Schauder's theorem does not guarantee (see the next section).

This method requires the existence of linear operator and a set of unit vectors satisfying some properties. Moreover, this method naturally leads to multiplicity results which are rather difficult to achieve with classical fixed point theorems. We point out that Clark's theorem is often applied to prove the existence and multiplicity of weak solutions of boundary problems (see, e.g., \cite{Corre2,MMO2015,MMOT2016}), but we belive that our work is among the first to apply it in the fixed point theory.
Our main existence theorems are applied to a Dirichlet boundary value problems associated to a second-order ordinary differential equation. This simple model was chosen only to illustrate the effectiveness of the new fixed point theorems.

Our first existence result reads as follows.
\begin{theorem}\label{thepf}
Let $H$ be a Hilbert space and $A: H\rightarrow H$  an odd compact potential operator satisfying assumption $(\mathcal{H})$. Assume  that there exist a linear operator $B_1$ on $H$, $r_1>0,$ and  $e_1\in H,$ with $\|e_1\|=1,$ such that
\begin{itemize}
\item[$(\mathcal{H}1)$] $\big(B_1(e_1),e_1\big)>1$,
\item[$(\mathcal{H}2)$] $\big(A(sr_1e_1),r_1e_1\big)\geq sr_1^2\big(B_1(e_1),e_1\big)$,  for all $s\in(0,1)$.
\end{itemize}
Then  operator $A$ has a pair of fixed points in $H\setminus\{0\}$.
\end{theorem}
Our second existence result reads as follows.
\begin{theorem}\label{theo2}
Let $H$ be a Hilbert space and $A: H\rightarrow H$  an odd compact potential operator satisfying assumption $(\mathcal{H}).$
Assume that there exists a linear self-adjoint operator $B_2$ on $H$ such that
\begin{itemize}
  \item[$(\mathcal{H}1)'$] there exist $e_2, e_3\in H$ with $\|e_i\|=1,\;i=2,3,$ and $(e_2,e_3)=0,$ satisfying
$$
\left\{
\begin{array}{ll}
\big(B_2(e_i),e_i\big)>1, \quad i=2,3.\\
\big(B_2(e_2),e_3\big)^2-\Big(1-\big(B_2(e_2),e_2\big)\Big)\Big(1-\big(B_2(e_3),e_3\big)\Big)<0,
\end{array}
\right.
$$
\item[$(\mathcal{H}2)'$] there exists a constant $r_2>0$ satisfying
$$
\big(A(su),u\big)\geq\big(B_2(su),u\big), \ \hbox{for all} \ \,u\in\partial \mathcal{B}(0,r_2)\cap<e_2,e_3>,\; \ \hbox{for all} \ \,s\in(0,1),
$$
where $\partial\mathcal{B}(0,r_2)$ denotes the boundary of the ball $\mathcal{B}(0,r_2)$ and $<e_2,e_3>$ is the subspace spanned by the vectors $e_2$ and $e_3$.
\end{itemize}
Then operator $A$ has two pairs of fixed points in $H\setminus\{0\}$.
\end{theorem}

We complete the introduction by an outline of the paper.
In Section \ref{s2}, we collect the necessary preliminary material.
In Section \ref{s3}, we prove the first existence result (Theorem \ref{thepf}).
In Section \ref{s4}, we prove the second existence result (Theorem \ref{theo2}).
Finally, in Section \ref{s5},  we give an application to a second-order
boundary value problem to illustrate our results.

\section{Preliminaries}\label{s2}

Let $E$ be a real Banach space and $\Sigma$ the class of all closed subsets $F\subset E\setminus\{0\}$ that are symmetric with respect
to the origin, i.e.,  $u\in F$ implies $-u\in F$.

\begin{definition}
Let $F\in\Sigma$. The Krasnosel'skii genus $\gamma(F)$ is defined as the least positive integer $n$ such that there is
an odd mapping $\varphi\in C(F, \mathbb{R}^n\setminus\{0\})$. If such $n$ does not exist, then we set $\gamma(F)=+\infty$. Moreover, by
definition, $\gamma(\emptyset)=0$.
\end{definition}

Next we shall present a result on the computation of the genus that will be used in this work.

\begin{proposition}\label{propo} (see \cite{rabin})
Let $F\subset E$, $\Omega$ be a bounded neighborhood of $0$ in $\mathbb{R}^N$, and assume that there exists an odd homeomorphism
$h\in C(F,\partial\Omega)$. Then $\gamma(F)=N$.
\end{proposition}

More details on the genus can be found  in \cite{ambro,castro,kavian,krasno}.

\begin{definition}
Let $J\in C^1(E,\mathbb{R})$. If any sequence $(u_n)\subset E$ for which $(J(u_n))$ is bounded and $J'(u_n)\rightarrow 0,$ as $n\rightarrow +\infty$ in $E',$ possesses a convergent subsequence, then we say that $J$ satisfies the Palais-Smale condition (denoted by (PS) condition).
\end{definition}

The following theorem, due to Clark  \cite{clark}, will be crucial in the proof of our existence results.

\begin{theorem}\label{clarke}
Let $J\in C^1(E,\mathbb{R})$ be a functional satisfying the (PS) condition. Assume further that:
  \item [(a)] $J$ is even and bounded from below,
  \item [(b)] there is a compact set $K\in\Sigma$ such that $\gamma(K)=k$ and $\sup\limits_{x\in K}J(x)<J(0)$.\\
Then $J$ possesses at least $k$ pairs of distinct critical points and their corresponding critical values are less than $J(0)$.
\end{theorem}
We point out that this result is a consequence of a basic multiplicity theorem involving an invariant functional under the
action of a compact topological group (see \cite{clark, rabin}).

Recall that a mapping is said to be compact if it maps bounded sets onto relatively compact sets. An operator $A: E\longrightarrow E'$ is called a potential operator, if there exists a G\^{a}teaux differentiable functional $T: E\longrightarrow\mathbb{R}$ such that $T'(x)=A(x)$, for every $x\in E$ (see \cite{Ch1979,groesen}), where $E'$ refers to the topological dual of $E$.
 Due to Avez \cite{avez}, we know that, for all $u\in E,$
$$
T(u)=\int_0^1 (A(su),u)_{E',E} \;ds.
$$
 Here $(\cdot,\cdot)_{E',E}$ refers to the duality pairing between $E$ and its topological dual $E'$.

In this paper, we shall assume that $\big(H, (\cdot, \cdot)\big)$ is a Hilbert space with $(\cdot, \cdot)$ denoting the scalar product on $H$ and $\|\cdot\|= \sqrt{(\cdot,\cdot)}$. Let $A$ be a compact operator satisfying $(\mathcal{H})$. Then there exist
 positive constants $c, b$ such that for all $u\in H$, we have
$$
 \|Au\|\leq c \|u\|^{\theta}+b.
$$
Indeed, by $(\mathcal{H})$ there exists some $R>0$ such that for all $u\in H$ with $\|u\|\geq R,$ we have $\|Au\|\leq c\|u\|^{\theta}$.
Let $\mathcal{B}(0,R)\subset H$ be the ball centered at the origin with radius $R.$ Since $A$ is a compact operator, it
follows that $\overline{A(\mathcal{B}(0,R))}$ is compact, whence $A(\mathcal{B}(0,R))$ is bounded, i.e. there exists $b>0$ such that
$$
 \|Au\|\leq b, \;\textrm{for all}\; u\in \mathcal{B}(0,R).
$$
Then Schauder's fixed point theorem applies and yields a solution lying in the ball, which is possibly the trivial fixed point.
On the other hand, our Theorem \ref{thepf} will provide existence of a nontrivial fixed point $u$. Since $A$ is odd, it follows that $-u$ is also a fixed point.

\section{Proof of Theorem~\ref{thepf}}\label{s3}

Consider the functional $J: H\rightarrow\mathbb{R}$ defined by
\begin{equation}\label{funct}
J(u)=\frac{1}{2}\|u\|^2-\int_0^1(A(su),u)\,ds,
\end{equation}
where $Tu=\int_0^1(A(su),u)\,ds.$ Let $X_1=\rm{span} \{e_1\}\subset H$ be the subspace spanned by the vector $e_1$ and consider the set
$$  K =\{u\in X_1: \|u\|=r_1\}
    = \{\lambda e_1: \|\lambda e_1\|=r_1\}
    = \{-r_1e_1, r_1e_1\}.
$$

\textbf{Step 1.}
It is clear that $K$ is symmetric. Since $X_1$ and $\R$ are isomorphic and $K$ and $S^0$ are homeomorphic,
 Proposition \ref{propo} implies that $\gamma(K)=1$; here $S^0$ refers to the
unit sphere in $\mathbb{R}$. Using hypotheses $(\mathcal{H}1)$ and $(\mathcal{H}2),$ 
we get
$$J(-r_1e_1)=J(r_1e_1)
=
\frac{1}{2}r_1^2-\int_0^1\big(A(sr_1e_1),r_1e_1\big)\,ds
$$
$$
\leq\frac{1}{2}r_1^2-\frac{1}{2}r_1^2\big(B_1(e_1),e_1\big)
=\frac{1}{2}\left(1-\big(B_1(e_1),e_1\big)\right)r_1^2<0.
$$
It follows that $$\sup\limits_{K}J(u)=\max\limits_{u\in\{-r_1e_1,r_1r_1\}}J(u)<0=J(0).$$

\textbf{Step 2.}
$J$ is bounded from below. Indeed,
\begin{equation}\label{lbound}
\begin{array}{rll}
  J(u)&\geq&\frac{1}{2}\|u\|^2-\int_0^1\|A(su)\|\|u\|\,ds\\
    &\geq&\frac{1}{2}\|u\|^2-\|u\|\int_0^1(c\|su\|^{\theta}+b)\,ds \\
    &\geq&\frac{1}{2}\|u\|^2-\frac{c}{\theta+1}\|u\|^{\theta+1}-b\|u\|,
    \end{array}
\end{equation}
and our claim follows from the fact that $\theta\in[0,1)$.

\textbf{Step 3.}
First, notice that since $A$ is a potential operator,  there exists a G\^{a}teaux differentiable functional $T: H \longrightarrow\mathbb{R}$ such that $T'=A$. More precisely, $$T(u)=\int_0^1<A(su),u>ds.$$ Hence $J(u)=\frac{1}{2}\|u\|^2-Tu$ and $J'(u)v=(u,v)-(Au,v)$, for all $v\in H$, that is $J'=I-A$.

Moreover, $J$ satisfies the (PS) condition. Indeed, let $(u_n)$ be a sequence in $H$ such that $J'(u_n)\rightarrow 0$ and $J(u_n)$ is bounded. From (\ref{lbound}), we get that $(u_n)$ is bounded from below in $H$. Since $A$ is compact, there exists a subsequence $(u_{n_k})\subset H$ such that $A(u_{n_k})\rightarrow v\in H$. Therefore $u_{n_k}\rightarrow v$ in $H$ for
$$
\!\|u_{n_k}-v\|\leq \|u_{n_k}-A(u_{n_k})\|+\|A(u_{n_k})-v\|=\! \|J'(u_{n_k})\|+\|A(u_{n_k})-v\|,
$$
and the right-hand side tends to $0$, as  $k\rightarrow+\infty$.

By Theorem \ref{clarke}, we conclude that $J$ has a pair of nontrivial critical points which are fixed points for the operator $A$.
This completes the proof of Theorem~\ref{thepf}.
\qed

\section{Proof of Theorem~\ref{theo2}}\label{s4}

We argue as in the proof of Theorem \ref{thepf}. Let $X_2=\rm{span} \{e_2,e_3\}$ be the subspace of $H$ spanned by the vectors $e_2$ and $e_3$ and consider the set:
$$
K'=\{u\in X_2: \|u\|=r_2\}=\{\alpha e_2+\beta e_3: \alpha^2+\beta^2=r_2^2\}.
$$
Clearly, $K'\in\sum$ and the homeomorphism 
$h:K' \to S^1,$ given by
$$ u=\alpha e_2+\beta e_3\mapsto(\frac{\alpha}{\sqrt{\alpha^2+\beta^2}} ,\frac{\beta}{\sqrt{\alpha^2+\beta^2}})$$ is odd. Now, Proposition \ref{propo} guarantees that $\gamma(K')=2$. Here, $S^1$ is the unit sphere in $\mathbb{R}^2$.

For $u\in K'$, assumption $(\mathcal{H}1)'$ now yields
\begin{eqnarray*}
J(u)&=&J(\alpha e_2+\beta e_3)
=\frac{1}{2}r_2^2-\int_0^1\big(A(su),u\big)\,ds\\
&\leq&\frac{1}{2}r_2^2-\int_0^1\big(B_2(su),u\big)\, ds
=\frac{1}{2}r_2^2-\frac{1}{2}\big(B_2(\alpha e_2+\beta e_3),\alpha e_2+\beta e_3\big) \\
&=&\frac{1}{2}r_2^2-\frac{1}{2}\alpha^2\big(B_2(e_2),e_2\big)-\frac{1}{2}\alpha\beta\big(B_2(e_2),e_3\big)
-\frac{1}{2}\alpha\beta\big(B_2(e_3),e_2\big)-\frac{1}{2}\beta^2\big(B_2(e_3),e_3\big)\\
&=&\frac{1}{2}\big(1-\big(B_2(e_2),e_2\big)\big)\alpha^2+\frac{1}{2}\big(1-\big(B_2(e_3),e_3\big)\big)\beta^2-\alpha\beta\big(B_2(e_2),e_3\big).
\end{eqnarray*}
Since the right-hand side term is strictly negative,
it follows  by  compactness of $K'$
 that  $\sup_{K'}J(u)<0=J(0)$.
The rest of the proof parallels that of Theorem \ref{thepf}.
This completes the proof of Theorem~\ref{theo2}.
\qed

\begin{remark}
If we attempt to extend the preceding results to the existence of $n$ pairs of fixed points, then we can consider the $(n-1)$-dimensional subspace
$$
K_n=\{u\in X_n=\langle e_1,e_2,...,e_n\rangle, \|u\|=r\}
=\{\alpha_1 e_1+\alpha_2 e_2+...+\alpha_ne_n, \sqrt{\Sigma_{i=1}^n\alpha_i^2}=r\}.
$$
Clearly, $K$ is symmetric. If we assume that $$\big(A(su),u\big)\geq\big(B(su),u\big),
\
\hbox{ for all}
\
u\in K_n, s\in (0,1),$$ where $B$ a linear self-adjoint operator, the question amounts to finding
 a sufficient condition for $\sup\limits_{K_n} J(u)$ to be negative.

Let $u\in K_n.$ Then, since $B$ is linear and self-adjoint,
we know that
\begin{eqnarray*}
J(u)&=&\frac{1}{2}r^2-\int_0^1\big( A(su),u\big)\,ds
\leq \frac{1}{2} r^2-\frac{1}{2}\big(B(\Sigma_{i=1}^n\alpha_ie_i),\Sigma_{i=1}^n\alpha_ie_i \big)\\
&=&\frac{1}{2} r^2-\frac{1}{2}\Sigma_{i=1}^{n}\Big(\Sigma_{j=1}^{n}\alpha_i\alpha_j(Be_i,e_j)\Big)\\
&=&\frac{1}{2}\Big[\underbrace{\alpha_1^2+\alpha_2^2+...+\alpha_n^2}_{r^2}-\Sigma_{i=1}^{n}\big(\Sigma_{j=1}^{n}\alpha_i\alpha_j(Be_i,e_j)\big].
\end{eqnarray*}
However, we do not know whether the term on the right-hand side is negative.
\end{remark}

 \section{Applications}\label{s5}

 \begin{example}
 Consider the following boundary value problem
\begin{eqnarray}\label{prb}
\left\{
\begin{array}{rll}
-u''(t)&=&f(t,u(t)), \quad t\in [0,1],\\
u(0)=u(1)&=&0,
\end{array}
\right.
\end{eqnarray}
where $f: [0,1]\times\mathbb{R}\rightarrow\mathbb{R}$ is a continuous function.
Clearly, solutions of problem (\ref{prb}) can be obtained as fixed points of the mapping $A$ defined on $H^1_0$ by
\begin{equation}\label{ope}
  Au(t)= \int_0^1G(t,s)f(s, u(s))\,ds,
\end{equation}
where
$$
G(t,s)=
\left\{
\begin{array}{ll}
t(1-s), & t\leq s, \\
s(1-t), & s\leq t
\end{array}
\right.
$$
and $H^1_0$ is the standard Hilbert space with the scalar product $(u,v)=\int_0^1u'(t)v'(t)dt$ and the corresponding norm $\Vert u\Vert=\left(\int_0^1\vert u'(t)\vert^2dt\right)^{1/2}.$
Then $A$ satisfies
\begin{equation}\label{}
  \left\{
  \begin{array}{ll}
    -(Au)''(t)= f(t,u(t)), & \hbox{} \\
    (Au)(0)=(Au)(1)=0. & \hbox{}
  \end{array}
\right.
\end{equation}
 \end{example}
 \begin{remark}
One can check that the operator $A: H^1_0\rightarrow H^1_0$ is compact (see \cite{bouc}).
\end{remark}

Next, we shall give an application to a second-order
boundary value problem.
\begin{theorem}\label{app2}
Assume that $f: [0,1]\times\mathbb{R}\rightarrow\mathbb{R}$ and positive functions $a_i\in L^1(0,1), \;a_i>0,\;\mbox{ a.e. }\,t\in(0,1)\;(i=2,3),$ and $a_1\in L^{\infty}(0,1),$ where $m=\essinf_{(0,1)}a_1>0$ and $M=\esssup_{(0,1)}a_1$, satisfy the following conditions:
\begin{itemize}
\item[$(D_1)$] there exists $r_1\in (0,1)$ such that $\frac{f(t,u)}{u}\geq a_1(t),$ for all $u\in[-r_1, r_1]\setminus\{0\}$,
\item[$(D_2)$] $f(t,u)\leq a_2(t)|u|^{\theta}+a_3(t),$ for all $u\in\mathbb{R}$ and  some $\theta\in [0,1)$,
\item[$(D_3)$] there exist $e_1, e_2\in H^1_0$ with $(e_1,e_2)=0$ and $\|e_1\|=\|e_2\|=1$ such that $m|e_i|_{L^2}>1$, $i=1,2$,
\item[$(D_4)$] $M^2+2\pi^2 m<\pi^4+m^2$.
\end{itemize}
Then problem (\ref{prb}) has two pairs of nontrivial solutions in $H^1_0$.
\end{theorem}

\begin{proof}
Let
\begin{equation}\label{Jf}
J(u)=\frac{1}{2}\Vert u\Vert^2-\int_0^1F(t,u)dt,
\end{equation}
where $F(t,u)=\int_{0}^{u}f(t,v)dv$. Using ($D_1$) together with the Lebesgue dominated convergence theorem, we can prove that $J\in C^1(H^1_0,\mathbb{R})$ and that for all $u,v\in H^1_0,$
\begin{eqnarray*}
  J'(u)(v) &=& \int_0^1u'v'\,dt-\int_0^1f(t,u(t))v(t)\,dt
           =\int_0^1u'v'\,dt+\int_0^1 (Au)''(t)v(t)\,dt\\
           &=&  \int_0^1(u'v'-(Au)'(t))v'(t)\,dt
           = (u,v)-(Au,v).
\end{eqnarray*}
Hence $A$ is a potential operator. Define
\begin{equation}\label{opB}
Bu(t)=\int_0^1G(t,s)a_1(s)u(s)\,ds.
\end{equation}
We shall apply  Theorem \ref{theo2}. The functional $J$ and operator $B$ are as given by (\ref{Jf}) and (\ref{opB}).
By ($D_3$), for $i=1, 2$, we have
\begin{eqnarray*}
\big(B(e_i),e_i\big) &=& \int_0^1\big(B(e_i)\big)'(t)_i'(t)dt
=-\int_0^1\big(B(e_i)\big)''(t)e_i(t)dt\\
&=&\int_0^1a_1(t)e_i^2(t)dt
\geq m|e_i|^2_{L^2}> 1.
\end{eqnarray*}

Using the Cauchy-Schwarz and the Poincar\'{e} inequalities together with assumption ($D_4$), we get the estimates
\begin{eqnarray*}
&&\big(B(e_1),e_2\big)^2-\Big(1-\big(B(e_1),e_1 \big)\Big)\Big(1-\big(B(e_2),e_2 \big)\Big)\\
\\
&=&\Big(\int_0^1\big(B(e_1)\big)'(t)e'_2(t)dt\Big)^2
-\Big(1-\int_0^1\big(B(e_1)\big)'(t)e'_1(t)dt\Big)\Big(1-\int_0^1\big(B(e_2)\big)'(t)e'_2(t)dt\Big)\\
\\
&=&\Big(\int_0^1\big(B(e_1)\big)''(t)e_2(t)dt\Big)^2
-\Big(1+\int_0^1\big(B(e_1)\big)''(t)e_1(t)dt\Big)\Big(1+\int_0^1\big(B(e_2)\big)''(t)e_2(t)dt\Big)\\
\\
&=&\Big(\int_0^1a_1(t)e_1(t)e_2(t)dt\Big)^2-\Big(1-\int_0^1a_1(t) e^2_1(t)dt\Big)\Big(1-\int_0^1(a_1(t)e^2_2(t)dt\Big)\\
\\
&\leq& M^2|e_1|^2_{L^2}|e_2|^2_{L^2}-\big(1-m|e_1|^2_{L^2}\big)\big(1-m|e_2|^2_{L^2}\big)\\
\\
&\leq& M^2|e_1|^2_{L^2}|e_2|^2_{L^2}-1+m|e_2|^2_{L^2}+m|e_1|^2_{L^2}-m^2|e_1|^2_{L^2}|e_2|^2_{L^2}\\
\\
&\leq&\frac{1}{\lambda_1^2}M^2+\frac{2}{\lambda_1}m-1-\frac{1}{\lambda_1^2}m^2=\frac{1}{\pi^4}M^2+\frac{2}{\pi^2}m-1-\frac{1}{\pi^4}m^2
<0,
\end{eqnarray*}
where $\lambda_1=\pi^2$ is the first eigenvalue of the linear Dirichlet problem:
$$
\left\{
\begin{array}{ll}
-u''(t)=\lambda u(t), &t\in [0,1], \\
u(0)=u(1)=0.
\end{array}
\right.
$$
Let $u\in H^1_0$ with $\|u\|=r_1$ and $s\in(0,1)$. Since $$|u|_{\infty}\leq \|u\|=r_1,$$ i.e., $-r_1\leq u\leq r_1$, we have, by $(D_1)$,
\begin{eqnarray*}
\big(A(su),u\big)-\big(B(su),u\big)&=&-\int_0^1\big(A(su)\big)''(t)u(t)\,dt+\int_0^1\big(B(su)\big)''(t)u(t)\,dt\\
&=&\int_0^1\frac{f(t,su(t))}{u(t)} u^2(t)\,dt-\int_0^1 a_1(t)sre^2(t)\,dt\\
&=&\int_0^1\left(\frac{f(t,su(t))}{u(t)}-a_1(t)\right) u^2(t)\,dt\geq 0.
\end{eqnarray*}
We can now deduce that conditions $(\mathcal{H}_1)'$ and $(\mathcal{H}_2)'$ are satisfied.
Regarding condition $(\mathcal{H})$, it is easy to see that it follows from condition $(D_2).$ Indeed, since the embedding $H^1_0\hookrightarrow C$ is continuous, we have
\begin{eqnarray*}
\|Au\|&=&\sup_{\|v\|_{H^1_0}\leq 1}|(Au,v)|
=\sup_{\|v\|_{H^1_0}\leq 1}\left\vert\int_0^1(Au)'(t)v'(t)\,dt\right\vert\\
&=&\sup_{\|v\|_{H^1_0}\leq 1}\left\vert\int_0^1-(Au)''(t)v(t)\,dt\right\vert
\leq\sup_{\|v\|_{H^1_0}\leq 1}\int_0^1|f(t,u(t)) v(t)|\,dt\\
&\leq&c\int_0^1(a_2(t)|u(t)|^{\theta}+a_3(t))\,dt
\leq c\|u\|^{\theta} \int_0^1a_2(t)\,dt+\int_0^1a_3(t)\,dt.
\end{eqnarray*}
Letting $c_1=\int_0^1a_2(t)\,dt$ and $c_2=\int_0^1a_3(t)\,dt$, we obtain
$$
\limsup_{\|u\|\rightarrow+\infty}\frac{\|Au\|}{\|u\|^{\theta}}<\infty,
$$
hence $(\mathcal{H})$ is satisfied.

Finally, Theorem \ref{theo2} guarantees that  operator $A$ has two pairs of fixed points in $H^1_0$. As a consequence, problem (\ref{prb}) admits two pairs of nontrivial solutions in $H^1_0$.
This completes the proof of Theorem~\ref{app2}.
\end{proof}
\begin{example}
For any $\theta\in[0,1)$ and  $r_1, t\in (0,1)$,
consider the following function
\begin{equation*}
f(t,u)=
    \left\{
      \begin{array}{ll}
        a_2(t) u^{\theta}, & \hbox{if $u\geq0$,} \\
        -a_2(t) (-u)^{\theta}, & \hbox{if $u<0$,}
      \end{array}
    \right.
\end{equation*}
where  $a_i,$ $ i\in\{1,2,3\},$ are defined as follows
\begin{itemize}
\item $a_1(t)=1+t$,
\item  $a_2(t)= r_1^{1-\theta}a_1(t),$  
\item $a_3(t)= t.$
\end{itemize}
Then function $f$ satisfies the hypotheses of Theorem \ref{app2}.
\end{example}

\subsection*{Acknowledgements}
Repov\v{s} was supported by the Slovenian Research and Innovation Agency grants P1-0292, N1-0278, N1-0114, J1-4031, J1-4001, and N1-0083.
We thank the anonymous referees for their constructive comments and suggestions.

\end{document}